\documentclass[11pt]{amsart}

\usepackage{amsthm, amsrefs, hyperref}

\usepackage{amsmath,amssymb}

\usepackage[margin=1.5in]{geometry}
\usepackage{libertine}
\usepackage[libertine]{newtxmath}

\usepackage[T1]{fontenc}

\usepackage[shortlabels]{enumitem}

\newtheorem{lemma}{Lemma}

\newtheorem{theorem}{Theorem}

\title[Limits of products of positive elements]{C*-algebras where every element is a limit of products of positive elements}
\author{Leonel Robert}
\address{Department of Mathematics, University of Louisiana at Lafayette, Lafayette, 70506, USA}
\email{lrobert@louisiana.edu}

\begin{document}
\begin{abstract}
We obtain a characterization of the unital C*-algebras with the property that every element is a limit of products of positive elements, thereby answering a question of Murphy and Phillips.	
\end{abstract}	

\maketitle	
	
Let $A$ be a unital C*-algebra. Let $A_+$ denote the set of positive elements of $A$. Let $P(A)$ denote  the set of finite products of elements in $A_+$.  Quinn showed in \cite{quinn} that if $A$ is an AF C*-algebra and has no finite dimensional representations, then $P(A)$ is  dense in $A$. In \cite{mur-phil}, Murphy and Phillips took up the systematic investigation of the C*-algebras with the property that $P(A)$ is dense in $A$, which they termed the  approximate positive factorization property (APFP). They obtained necessary conditions for the APFP to hold, confirmed the APFP for various classes of C*-algebras, and raised the following question (\cite[Question 4.12]{mur-phil}): If $A$ is an infinite dimensional simple unital C*-algebra of real rank zero, stable rank one, and trivial $K_1$-group, does $A$ have the  APFP? In the theorem below 
we answer this question affirmatively by establishing necessary and sufficient conditions for the APFP to hold.

\begin{theorem}\label{main}
Let $A$ be a unital C*-algebra. The following are equivalent:
\begin{enumerate}[(i)]
\item
$P(A)$ is dense in $A$, i.e., $A$ has the approximate positive factorization property.
\item
$A$ has no nonzero finite dimensional representations, the stable rank of $A$ is one, the $K_1$-group is trivial, and the map  $\rho_A\colon K_0(A)\to \mathrm{Aff}(T(A))$, from the $K_0$-group of $A$ to the continuous affine functions on the trace space of $A$, has dense range.
\end{enumerate}
\end{theorem}

Let us recall the meaning of the properties in (ii). 

A unital C*-algebra $A$ is said to have stable rank one if 
its invertible elements form a dense set.  

By trivial $K_1$-group we mean $K_1(A)=0$, where $K_1(A)$ is the abelian group associated to $A$ via K-theory. 

Let $T(A)$ denote the convex set of tracial states on $A$ (endowed with the weak* topology). Let us regard traces extended to $M_\infty(A)$---the infinite matrices on $A$ with finitely many nonzero entries---by setting 
$\tau((a_{ij})_{i,j})=\tau(\sum_{i=1}^\infty a_{ii})$. Given a tracial state $\tau$ and a $K_0$-group element $[p]-[q]\in K_0(A)$, define their pairing as $([p]-[q],\tau)\mapsto \tau(p)-\tau(q)$. Fixing $[p]-[q]$ and varying $\tau$, we get a continuous real-valued affine function on $T(A)$, i.e., an element of $\mathrm{Aff}(T(A))$. Let us denote by   $\rho_A\colon K_0(A)\to \mathrm{Aff}(T(A))$ the mapping induced by this pairing: 
\[
\rho_A([p]-[q])(\tau)=\tau(p)-\tau(q).
\]

It is known that if $A$ is a simple C*-algebra of real rank zero, then the range of $\rho_A$ is dense in $\mathrm{Aff}(T(A))$ (e.g., see \cite[Theorem 4.1]{Dur}). Thus,  Theorem \ref{main} indeed answers Murphy and Phillips' question affirmatively.

In the second theorem of this note (which we prove first), we obtain a description of the invertible elements in $\overline{P(A)}$ which, in a sense,  extends the description as ``elements of positive determinant", valid in the case of  matrices. Let $\widetilde P(A)=\overline{P(A)}\cap GL(A)$, where  $GL(A)$ denotes the invertible elements of $A$. Let $U(A)$ denote the unitary group of $A$, and $U_0(A)$ the connected component of 1 in $U(A)$. Let $DU_0(A)$ denote the commutator subgroup of $U_0(A)$. 

\begin{theorem}\label{main2}
	Let $A$ be a unital C*-algebra. Then $\widetilde P(A)$ is the set of invertible elements whose unitary part in the polar decomposition belongs to $\overline{DU_0(A)}$, i.e.,  the elements of the form $ua$, with $u\in \overline{DU_0(A)}$ and $a\in A_+$ positive and invertible.	
\end{theorem}

\section{Proofs of Theorems \ref{main} and \ref{main2}}
We will use some facts around the de la Harpe-Skandalis determinant. We introduce here necessary notation, but refer the reader to \cite{delaHarpe-skandalis} and \cite{thomsen} for background on this topic. 

Let $\overline{[A,A]}$ denote the closure of the linear span of the commutators $[x,y]:=xy-yx$, with $x,y\in A$. 
Let us regard $A/\overline{[A,A]}$ as a Banach space under the quotient norm. Let $T\colon A\to A/\overline{[A,A]}$ denote the quotient map (or ``universal trace"). Regard $T$ extended to $M_\infty(A)$ by setting $T((a_{ij})_{i,j})=T(\sum_{i=1}^\infty a_{ii})$.

Let $GL_0^\infty(A)$ denote the inductive limit of the groups $GL_0(M_n(A))$, $n\in \mathbb N$. Given a smooth path $\alpha\colon [t_1,t_2]\to GL_0^\infty(A)$, its de la Harpe-Skandalis determinant  is defined as 
\[
\widetilde \Delta_T(\alpha)=\int_{t_1}^{t_2} T(\alpha'(t)\alpha^{-1}(t))dt\in A/\overline{[A,A]}.
\]
If two paths $(a_1(t))_t$ and $(a_2(t)_t)$  both connect $1$ to an invertible element $a\in GL_0^\infty(A)$,
then the values of $\widetilde\Delta((a_1(t))_t)$ and $\widetilde\Delta((a_2(t))_t)$ differ by at most an element
of the abelian subgroup 
\[
2\pi i\{T(p)-T(q):p,q\in M_\infty(A)\hbox{ projections}\}.
\] 
Let's denote this subgroup by $2\pi iT(K_0(A))$). Given an invertible element $a\in GL_0^\infty(A)$, its de la Harpe-Skandalis determinant $\Delta_T(a)$ is defined as the image of $\widetilde\Delta_T((a(t)_t))$ in the quotient $(A/\overline{[A,A]})/2\pi iT(K_0(A))$, where $t\mapsto a(t)$ is any path in $GL_0^\infty(A)$ connecting $1$ to $a$.

\begin{proof}[Proof of Theorem \ref{main2}]
	Let $Q$ denote the set of invertible elements whose unitary part in the polar decomposition belongs to $\overline{DU_0(A)}$. Thus, $Q$ is defined by the condition $x|x|^{-1}\in  \overline{DU_0(A)}$. By the continuity of the operations involved in this condition, it is clear  that $Q$ is a norm closed subset of  $GL(A)$ (in fact, of $GL_0(A)$). 
	
	Let us prove that $Q$ is contained in $\widetilde P(A)$. By the form of the polar decomposition of the elements in $Q$,  it will suffice to show that  $DU_0(A)$ is contained in $\widetilde P(A)$. In fact, we will show that $DGL_0(A)$ is contained in $\widetilde P(A)$. It is easily checked, using polar decompositions, that $P(A)\cap GL_0(A)$  is invariant under similarity, i.e., it is a normal subgroup of $GL_0(A)$. Thus $\widetilde P(A)$ is a closed normal subgroup of $GL_0(A)$. By \cite[Theorem 1.1]{normal}, to show that  $\widetilde P(A)$ contains $DGL_0(A)$ it suffices to show the equality of closed two-sided ideals
	\[
	\mathrm{Ideal}([A,A])=\mathrm{Ideal}([\widetilde P(A),A]).
	\]
	Let $I$ be the ideal on the right hand side. It is clearly contained on the left hand side.
	Observe that  the image of $\widetilde P(A)$ in the quotient $A/I$ is central. Since the invertible positive elements of $A$ map onto the invertible positive elements of $A/I$, and the latter span $A/I$, it follows that $A/I$ is commutative. Thus,  $\mathrm{Ideal}([A,A])\subseteq I$, as desired. We thus conclude by \cite[Theorem 1.1]{normal} that $DGL_0(A)\subseteq \widetilde P(A)$.
	
	For the inclusion $\widetilde P(A)\subseteq Q$, it will suffice to show that $Q$ is a group, as it is a closed set (in $GL_0(A)$) and contains all invertible positive elements.  Let $u_1a_1$ and $u_2a_2$ be in $Q$, with $u_1,u_2\in \overline{DU_0(A)}$ and $a_1,a_2\in A_+$. Then,
	\[
	u_1a_1u_2a_2=u_1u_2 (u_2^*a_1u_2)a_2.
	\]
	Set $a_1'=u_2^*a_1u_2$. Let $a_1'a_2=w|a_1a_2|$ be the polar decomposition of $a_1'a_2$. Then
	\[
	u_1a_1u_2a_2=u_1u_2w |a_1a_2|.
	\]
	We will be done once we have shown that $w\in \overline{DU_0(A)}$. That is, we must show that given two invertible positive elements $a,b$, the unitary $u=ab|ab|^{-1/2}$ (in the polar decomposition of $ab$) belongs to $\overline{DU_0(A)}$. Write $a=e^{c}$ and $b=e^{d}$, with $c,d\in A$ selfadjoint.  
	Define $u_t=e^{tc}e^{td}|e^{td}e^{tc}|^{-1/2}$, for $t\in [0,1]$. This is a path of unitaries connecting 1 to $u$. Let us  calculate its determinant: 
	\begin{align*}
	\widetilde{\Delta}_T((u_t)_t) &=\widetilde\Delta_T((e^{tc})_t)+\widetilde\Delta_T((e^{td})_t)-\widetilde\Delta_T((|e^{tc}e^{td}|)_t)\\
	&=T(c)+T(d)-\widetilde\Delta_T((|e^{tc}e^{td}|)_t).
	\end{align*}
	Since  the path 
	\[
	t\mapsto |e^{tc}e^{td}|^2=e^{td}e^{2tc}e^{td}\] 
	has determinant $2(T(c)+T(d))$, the path $t\mapsto |e^{tc}e^{td}|$ has determinant $T(c)+T(d)$. Hence, from the above calculation we get $\widetilde\Delta_T((u_t)_t)=0$. By (the proof of) \cite[Lemma 3]{delaHarpe-skandalis},  
	\[
	u=\prod_{k=1}^N e^{ih_k},\] 
	where $h_k$ is selfadjoint for all $k$ and 
	$\sum_{k=1}^N h_k\in \overline{[A,A]}$. (From the proof of \cite[Lemma 3]{delaHarpe-skandalis}, $e^{ih_k}=u_{t_{k-1}}^{-1}u_{t_k}$ for all $k$, where $0=t_0<t_1<\cdots<t_N=1$ is a sufficiently  fine partition of the interval $[0,1]$.) Now $\prod_{k=1}^N e^{ih_k}$ belongs to $\overline{DU_0(A)}$,  by \cite[Theorem 6.2 (b)]{normal} (alternatively, by the proof of \cite[Proposition 4]{delaHarpe-skandalis}).	
\end{proof}		

In the proof of Theorem \ref{main} we rely on some results by Thomsen from \cite{thomsen}. We isolate these facts in a lemma:

\begin{lemma}\label{fromthomsen}
	Let $A$ be a unital C*-algebra. If $U_0(A)=\overline{DU_0(A)}$ then $\rho_A\colon K_0(A)\to \mathrm{Aff}(T(A))$ has dense range. Conversely, if $\rho_A$ has dense range and $A$ has stable rank one, then $U_0(A)=\overline{DU_0(A)}$. 	
\end{lemma}

\begin{proof}
We maintain the notation used in \cite[Section 3]{thomsen}. By \cite[Theorem 3.2]{thomsen}, 
\begin{equation}\label{thomseniso}
U_0(A)/\overline{DU_0(A)}\cong \mathrm{Aff}(T(A))/\overline{\Delta_1^0(\pi_1(U_0(A)))},
\end{equation} 
where $\Delta_1^0(\pi_1(U_0(A))$ is a subgroup of $\rho_A(K_0(A))$.  Thus, from $U_0(A)=\overline{DU_0(A)}$ we get that $\Delta_1^0(\pi_1(U_0(A))$ is dense in $\mathrm{Aff}(T(A))$, and a fortiori, $\rho_A$
has dense range. (The map $\Delta_1^0$ is defined in terms of the de la Harpe-Skadalis determinant for paths recalled above. More specifically, given a loop $\alpha$ in $U_0(A)$ starting and ending in 1, one computes $h=\frac{1}{2\pi i}\widetilde \Delta_T(\alpha)\in A/\overline{[A,A]}$, and  takes the affine function $\widehat h\in \mathrm{Aff}(T(A))$ induced by this element. Then $\Delta_1^0([\alpha]):=\widehat h$.
By the homotopy invariance of the determinant, and the isomorphism $\pi_1(U_0^\infty(A))\cong K_0(A)$ from Bott periodicity, $\Delta_1^0([\alpha])\in \rho_A(K_0(A))$.)
	
Suppose now that  $\rho_A$ has dense range and $A$ has stable rank one. By the isomorphism \eqref{thomseniso}, 
 to show that $U_0(A)=\overline{DU_0(A)}$ it suffices to show that $\Delta_1^0(\pi_1(U_0(A)))=\rho_A(K_0(A))$.
For C*-algebras of stable rank one, the inclusion of $U_0(A)$ in $U_0^\infty(A)$ induces an isomorphism  of homotopy groups $\pi_1(U_0(A))\cong \pi_1(U_0^\infty(A))$.  This follows from combining \cite[Corollary 8.6]{rieffel}, which says if $A$ has stable rank one then $C(\mathbb{T},A)$ has connected stable rank at most two, and \cite[Proposition 2.6]{rieffel} applied to $C(\mathbb{T},A)$. Hence,
	$\pi_1(U_0(A))\cong \pi_1(U_0^\infty(A))\cong K_0(A)$.  The subgroup $\Delta_1^0(\pi_1(U_0(A)))$  is then precisely $\rho_A(K_0(A))$. 
\end{proof}

\begin{proof}[Proof of Theorem \ref{main}]
(i)$\Rightarrow$(ii). In fact, this implication is already present in \cite{mur-phil}. Suppose that $A$ has the APFP. By \cite[Proposition 2.7]{mur-phil},  $A$ has no finite dimensional representations,  and by \cite[Theorem 2.4]{mur-phil}, it has stable rank one and trivial $K_1$-group. (These are also fairly straightforward  consequences of the APFP.) The density of the range of $\rho_A$ is essentially proven in \cite[Lemma 2.8]{mur-phil}, but since that result is stated differently, and in less generality, we include an argument here. Let $u\in U_0(A)$. Then $u$ is the limit of a sequence of elements in $\widetilde P(A)$. By Theorem \ref{main2}, they have the form $w_na_n$ with $w_n\in \overline{DU_0(A)}$ and $a_n\in A_+$. Comparing unitary parts, we get that $w_n\to u$. It follows that $U_0(A)=\overline{DU_0(A)}$. Hence, by Lemma \ref{fromthomsen}, $\rho_A$ has dense range.

(ii)$\Rightarrow$(i). Since $A$ has stable rank one, it suffices to approximate invertible elements of $A$ by elements of $P(A)$. By polar decomposition, every invertible element has the form $ua$, where $u$ is a unitary and $a\in A_+$. By the description of $\widetilde P(A)$ from Theorem \ref{main2}, we will be done once we have shown that $U(A)=\overline{DU_0(A)}$.
Now, by the $K_1$-injectivity/surjectivity of stable rank one C*-algebras (\cite[Theorem 2.10]{rieffel}), $U(A)/U_0(A)\cong K_1(A)$. Since we have assumed that $K_1(A)=0$, we get that $U(A)=U_0(A)$. Thus, it remains to show that $U_0(A)=\overline{DU_0(A)}$. This follows from Lemma \ref{fromthomsen}.  
\end{proof}

For a nonunital C*-algebra $A$, Murphy and Phillips define the approximate positive factorization property  by asking that $1+P(A)$ be dense in $1+A$, in the unitization of $A$. A different, weaker, property is obtained by simply asking that $P(A)$ be dense in $A$. We have not  pursued the investigation of these approximate positive factorization properties in the nonunital setting.

\end{document}